\documentclass[11pt]{article}

\usepackage{amssymb}
\usepackage{amsmath}
\usepackage{a4wide}
\usepackage{amsthm}
\usepackage{hyperref}
\theoremstyle{plain}
\newtheorem{thm}{Theorem}[section] % reset theorem numbering for each chapter
\newtheorem{defn}[thm]{Definition} % definition numbers are dependent on theorem numbers
\newtheorem{lem}[thm]{Lemma} % same for lemma
\newtheorem{pro}[thm]{Proposition}

\theoremstyle{definition}

\newtheorem{rem}[thm]{Remark}

\begin{document}
	\begin{center}
		\section*{On the extension and kernels of signed bimeasures and their role in stochastic integration}
		\subsection*{Riccardo Passeggeri\footnote{\noindent Department of Mathematics, Imperial College London. Email: riccardo.passeggeri@imperial.ac.uk}}
		\today
	\end{center}

 			\begin{abstract}
 			In this work we provide a necessary and sufficient condition for the extension of signed bimeasures on $\delta$-rings and for the existence of relative kernels. This result generalises the construction method of regular conditional probabilities to the more general setting of (not necessarily finite) signed measures lying on not necessarily measurable spaces. Building on this result, we obtain the most general theory of stochastic integrals based on real-valued random measures, thus extending and generalising the whole integration theory developed in the celebrated Rajput and Rosinski's paper (\textit{Probab.~Theory Relat.~Fields}, \textbf{82} (1989) 451-487) and the recent results by Passeggeri (\textit{Stoch.~Process.~Their Appl.}, \textbf{130}, (3), (2020), 1735-1791).\\\\
 				\textbf{Keywords:} independently scattered random measure, signed bimeasure, regular conditional probability, Markov kernel,  quasi-infinite divisibility, stochastic integral.\\\\
 			 \textbf{MSC (2020):} 60G57, 60A10, 60E07, 28A35, 28A50, 28C15
 			\end{abstract}

\section{Introduction}
Bilinear measure theory has a long history and its origins date back to the work of Fr\'{e}chet in 1915 \cite{Frechet}. In this work, Fr\'{e}chet characterized the bounded bilinear functionals on $C[0,1]$, which are bimeasures when identified as set functions \cite{Ylinen}. Across the decades there have been numerous works on bilinear (and multilinear) measure theory. A major example is given by the work \cite{Grothe} by Grothendieck, where he introduced what is now called the Grothendieck's inequality, namely what he described as ``the fundamental theorem in the metric theory of tensor products". Further relevant works, which include also the efforts of obtaining the Riesz-Markov-Kakutani representation theorem in the bilinear and multilinear case, are \cite{Blei,Blei2,Bowers,Dobrakov,Horo,Karni,PrekopaI,Ylinen,Ylinen2}, among others.

However, all these works focused on algebras or $\sigma$-algebras of sets, but not on $\delta$-rings. Recall that a non-empty collection of sets is called a $\delta$-ring if it is closed under union, relative complementation, and countable intersection. Further, recall that any $\sigma$-algebra is a $\sigma$-ring and any $\sigma$-ring is a $\delta$-ring, but the reverse it is not true, namely not every $\delta$-ring is a $\sigma$-ring and not every $\sigma$-ring is a $\sigma$-algebra. The choice of investigating the $\delta$-rings framework, apart from an intrinsic interest, comes from the fact that this is one of the topological framework mostly used in the field of probability and stochastic analysis. 

For example, this is the framework at the core of the theory of random measures (\cite{Kallenberg2}). In the first section of the first chapter of the recent book of Kallenberg \cite{Kallenberg2}, the author introduces the notion of localising ring, which is a ring closed under countable intersections, thus a $\delta$-ring. The localising ring is then used the build the whole theory of (positive) random measures. Moreover, $\delta$-rings are the topological framework of certain classes of stochastic processes, in particular infinitely divisible (ID) stochastic processes. This class of processes contain some of the most studied processes: Brownian motions, Poisson processes, L\'{e}vy processes. The topological framework of the Rajput and Rosinski's paper \cite{RajRos}, where spectral representations for these processes are obtained, is the $\delta$-ring. Their work is still nowadays one of the main references for ID processes and one of the initial motivations of our work was to generalise their work.

Infinitely divisible (ID) distributions constitute one of the most important classes of probability distributions. Their investigation dates back to the work of L\'{e}vy, Kolmogorov and De Finetti. One of their most attractive properties is that their characteristic function have a unique explicit formulation, called the L\'{e}vy-Khintchine formulation, in terms of three mathematical objects. These are the drift, which is a real valued constant, the Gaussian component, which is a non-negative constant, and the L\'{e}vy measure, which is a measure on $\mathbb{R}$ satisfying an integrability condition and with no mass at $\{0\}$. Gaussian and Poisson distributions are examples of this class. 

In 2018, in \cite{LPS} Sato, Lindner and Pan introduced the class of quasi-infinitely divisible (QID) distributions. A QID random variable is defined as follows: a random variable $X$ is QID if and only if there exist two ID random variables $Y$ and $Z$ s.t.~$X+Y\stackrel{d}{=}Z$ and $Y$ is independent of $X$. Like ID distributions, QID distributions posses a  L\'{e}vy-Khintchine formulation where the L\'{e}vy measure is now allowed to take negative values. Any ID distribution is QID, but the converse is not always true.

In \cite{LPS}, the authors show that QID distributions are dense in the space of all probability distributions with respect to weak convergence, that distributions with a point mass greater than $1/2$ are QID, and that distributions that are concentrated on the integers are QID if and only if their characteristic functions have no zeros. In \cite{Pass}, the QID framework is extended to real-valued random measures and stochastic processes. The work \cite{Pass} also represents the first attempt to extend the celebrated Rajput and Rosinski's 1989 paper \cite{RajRos} to the QID framework. Furthermore, QID distributions have already shown to have an impact in various fields: from mathematical physics \cite{Physics2,Physics1} to number theory \cite{Naka,Naka2}, and from Bayesian analysis \cite{PassRM} to insurance mathematics \cite{Zhang}. Building on QID results, Cram\'{e}r-Wold devices for infinite and quasi-infinite divisibility have been recently proved in \cite{Khartov,Lindner}.

The first main contribution of this paper is a general measure theoretical result on the extension of signed bimeasures on $\delta$-rings. In particular, consider a bimeasure on the Cartesian product of $(X,\mathcal{B})$ and $(T,\mathcal{A})$, where the former is a Borel measurable space and the latter is such that $T$ is an arbitrary non-empty set and $\mathcal{A}$ is a $\delta$-ring with the additional condition that there exists an increasing sequence of sets $T_{1},T_{2},\dots \in {\mathcal {A}}$ s.t.~$\bigcup _{n\in \mathbb {N} }T_{n}=T$. Then, the result states that there exists a necessary and sufficient condition for the existence of a unique signed measure on the $\delta$-ring $\bigcup_{D\in\mathcal{A}}(\mathcal{A}\cap D)\otimes\mathcal{B}$, and that its unique Jordan decomposition uniquely extend to measures on the $\sigma$-algebra $\sigma(\mathcal{A})\otimes\mathcal{B}$. Moreover, the result provides unique (sub-Markovian) kernels for all these measures. 

Theorem 5.18 in \cite{Pass}, which is main achievement of \cite{Pass}, provides this result only for $\sigma$-algebras. This is a severe limitation as the topological framework of \cite{Pass}. which is the one of Rajput and Rosinski's paper \cite{RajRos}, is based on $\delta$-rings. In this paper, we succeed in extending Theorem 5.18 of \cite{Pass} and obtain the result for the right and more general topological framework of $\delta$-rings. This is our first main contribution.

We note that this result also extends a classical measure theoretical result by Horowitz in \cite{Horo} and can be seen as the generalisation of the construction method of regular conditional probabilities when the probability measure is an extended (\textit{i.e.}~not necessarily finite) signed measure.

Building on this first main contribution, we are able to obtain our second main contribution: we generalise, unify and simplify the theory of stochastic integrals based on random measures, thus extending the whole integration theory developed in Rajput and Rosinski's 1989 paper \cite{RajRos} and the results of \cite{Pass}. In particular, this paper represents the realization of the main goal envisioned and attempted in \cite{Pass} and it provides the most general integration theory with respect to random measures.

In particular, in \cite{Pass} the author generalises the results of Rajput and Rosinski's work \cite{RajRos} to the more general QID framework -- since any ID distribution is QID, but not vice versa. However, this is done under various restrictive assumptions (see Section 5 in \cite{Pass}). In this paper, building on our first contribution, we are able obtain the whole QID integration theory under one single assumption which is weaker than all the assumptions of \cite{Pass} combined together. More importantly, this assumption is \textit{always satisfied} in the ID setting and is unavoidable since it comes directly from the necessary and sufficient condition on the extension of signed bimeasures. This allows us to affirm that in this paper we provide a true and complete extension of the results in the Rajput and Rosinski's 1989 paper \cite{RajRos} and in \cite{Pass}, thus obtaining the most complete integration theory with respect to real-valued random measure.

As mentioned before, the topological framework of this paper is central in the theory of positive random measures (see \cite{Kallenberg2}). Thus, the results presented in this work are fundamental for developing the theory of random signed measures for which only sparse and specific results are known, precisely because of the lack of general measure theoretical results for signed measures in the literature. The interest on random signed measures is not new (see Daley and Vere-Jones' book \cite{Daley}), but it is sharply increasing due to modelling needs in Bayesian non-parametric analysis (see \cite{Nau,PassRSM}, among others).

The paper is structured as follows. Section \ref{Sec-Notation} concerns with the notations and some preliminaries. In Section \ref{Sec-Spectral}, the mentioned general measure theoretical result is presented (see Theorem \ref{pr-monster}). In Section \ref{Sec-Integration}, building on this result we derive the integration theory for QID random measures: L\'{e}vy-Khintchine formulations, integrability conditions, and continuity properties for QID stochastic integrals.
\section{Notation and Preliminaries}\label{Sec-Notation}
By a measure on a measurable space $(X,\mathcal{G})$ we always mean a positive measure on $(X,\mathcal{G})$, \textit{i.e.}~a $[0,\infty]$-valued $\sigma$-additive set function on $\mathcal{G}$ that assigns the value $0$ to the empty set. For a non-empty set $X$, by $\mathcal{B}(X)$ we mean the Borel $\sigma$-algebra of $X$, unless stated differently. The law and the characteristic function of a random variable $X$ will be denoted by $\mathcal{L}(X)$ and by $\hat{\mathcal{L}}(X)$, respectively. For two measurable spaces $(X,\mathcal{G})$ and $(Y,\mathcal{F})$, we denote by $\mathcal{G}\otimes\mathcal{F}$ the product $\sigma$-algebra of $\mathcal{G}$ and $\mathcal{F}$, and by $\mathcal{G}\times\mathcal{F}$ their Cartesian product. Let us recall some definitions.

\begin{defn}[extended signed measure]\label{Def-signedmeasure}
	Given a measurable space $(X, \Sigma)$, that is, a set $X$ with a $\sigma$-algebra $\Sigma$ on it, an extended signed measure is a function $\ \mu :\Sigma \to {\mathbb {R}}\cup \{\infty ,-\infty \}$ s.t.~$\mu (\emptyset )=0$ and $\mu$ is $\sigma$-additive, that is, it satisfies the equality $ \mu \left(\bigcup _{{n=1}}^{\infty }A_{n}\right)=\sum _{{n=1}}^{\infty }\mu (A_{n})$ where the series on the right must converge in ${\mathbb {R}}\cup \{\infty ,-\infty \}$ absolutely (namely the value of the series is independent of the order of its elements), for any sequence $A_{1}, A_{2},...$ of disjoint sets in $\Sigma$.
\end{defn}
\noindent As a consequence any extended signed measure can take plus or minus infinity as value, but not both. In this work, we use the term `signed measure' for an extended signed measure. Further, the \textit{total variation} of a signed measure $\mu$ is defined as the measure $|\mu|:\Sigma\rightarrow [0, \infty]$ defined by
\begin{equation}\label{def-totalvariation}
|\mu|(A):=\sup\sum_{j=1}^{\infty}|\mu(A_{j})|,
\end{equation}
where the supremum is taken over all the partitions $\{A_{j}\}$ of $A\in\Sigma$. The total variation $|\mu|$ is finite if and only if $\mu$ is finite. Let us recall the definition of a signed bimeasure.
\begin{defn}[Signed bimeasure]
	Let $(X,\Sigma)$ and $(Y,\Gamma)$ be two measurable spaces. A \textnormal{signed bimeasure} is a function $M:\Sigma\times\Gamma\rightarrow[-\infty,\infty]$ such that:
	\\\textnormal{(i)} the function $A\rightarrow M(A,B)$ is a signed measure on $\Sigma$ for every $B\in\Gamma$,
	\\\textnormal{(i)} the function $B\rightarrow M(A,B)$ is a signed measure on $\Gamma$ for every $A\in\Sigma$.
	
\end{defn}
\noindent Given a signed bimeasure $G$ on $\Sigma\times\Gamma$, we denote by $G^{+}$ and $G^{-}$ the Jordan decomposition of $B\mapsto G(A,B)$ for fixed $A\in\Sigma$, and by $G_{+}$ and $G_{-}$ the Jordan decomposition of $A\mapsto M(A,B)$ for fixed $B\in\Gamma$. 

We use the term `measure' and `signed measure' not only in the case of $\sigma$-algebra, but also in the case of rings, as follows.
\begin{defn}[Signed measure on a ring] \label{Def-signedmeasure-ring} A set function $\mu(A)$ defined on the elements of a ring $\mathcal{R}$ with values in $[-\infty,\infty]$ will be called a signed measure, if $\mu(\emptyset)=0$ and if for every sequence $A_{1},A_{2}, . . .$ of disjoint sets of $\mathcal{R}$ for which $A=\bigcup_{k=1}^{\infty}A_{k}\in\mathcal{R}$ we have
	\begin{equation}\label{measure on a ring}
	\mu(A)=\sum_{k=1}^{\infty}\mu(A_{k})
	\end{equation}
	and the relation (\ref{measure on a ring}) holds absolutely (namely independent of the order of its elements).
\end{defn}
\noindent Similarly, it is possible to extend the definition of bimeasures on rings. Let us now recall the celebrated Carath\'{e}odory's extension theorem.
\begin{thm}
	[Carath\'{e}odory's extension theorem, see Theorem 1.41 in \cite{Klenke}] Let $\mu$ be a measure on a ring $\mathcal{R}$ of subsets of a space $X$. Assume that $\mu$ is $\sigma$-finite (\textit{i.e.}~that there exists $S_{1},S_{2},...\in\mathcal{R}$ such that $X=\cup_{n=1}^{\infty}S_{n}$ and that $\mu(S_{n})<\infty$ for every $n\in\mathbb{N}$) then there exists a unique $\sigma$-finite measure $\bar{\mu}$ on $\sigma(\mathcal{R})$ such that $\mu(A)=\bar{\mu}(A)$ for all $A\in\mathcal{R}$.
\end{thm}
Denote by $S$ an arbitrary non-empty set. Let $\mathcal{S}$ be a $\delta$-ring with the additional condition that there exists an increasing sequence of sets $S_{1},S_{2},\dots \in {\mathcal {S}}$ s.t.~$\bigcup _{n\in \mathbb {N} }S_{n}=S$. 
\begin{defn}[random measure]\label{def-rm-as-in} Let $\Lambda = \{\Lambda(A):\,\, A\in\mathcal{S}\}$ be a real valued stochastic process defined on some probability space $(\Omega,\mathcal{F},\mathbb{P})$. We call $\Lambda$ to be a \textnormal{random measure}, if, for every sequence $\{A_{n}\}$ of disjoint sets in $\mathcal{S}$, the random variables $\Lambda(A_{n})$, $n= 1, 2, ...,$ are independent, and, if $\bigcup_{n=1}^{\infty}A_{n}\in\mathcal{S}$, then we have $\Lambda(\bigcup_{n=1}^{\infty}A_{n})=\sum_{n=1}^{\infty}\Lambda(A_{n})$ a.s. (where the series is assumed to converge almost surely).
\end{defn}
Notice that a more correct, but way more tedious, name for a random measure is `independently scattered real-valued completely additive stochastic set function on $\mathcal{S}$'. We remark that random measures are called sometimes called random noises, see \textit{e.g.}~\cite{SamTaq}, to distinguish from the random measures as defined in \cite{Kallenberg2}.

We recall the following result from Pr\'{e}kopa's works.
\begin{thm}[Theorem 2.1 in \cite{PrekopaI}]\label{TheoremPrekopa}
	In order that a finitely additive set function $\xi(A)$ defined on the elements of the ring $\mathcal{R}$ should be countably additive it is necessary and sufficient that, for every non-increasing sequence of sets $B_{1},B_{2},...$ with $B_{k}\in\mathcal{R}$ $(k = 1, 2, . . .)$ and $B_{n}\searrow\emptyset$, $\xi(B_{n})\stackrel{p}{\rightarrow}0$ as $n\rightarrow\infty$.
\end{thm}

Now, we introduce the concept of a quasi-L\'{e}vy type measure. We start with the following definition, which we recall from \cite{LPS}:
\begin{defn}\label{def1}
	Let $\mathcal{B}_{r}(\mathbb{R}):=\{B\in\mathcal{B}(\mathbb{R})| B \cap(-r, r) = \emptyset\}$ for $r > 0$ and $\mathcal{B}_{0}(\mathbb{R}):= \bigcup_{r>0} \mathcal{B}_{r}(\mathbb{R})$ be the class of all
	Borel sets that are bounded away from zero. Let $\nu : \mathcal{B}_{0}(\mathbb{R})\rightarrow\mathbb{R}$ be a function such that
	$\nu_{|\mathcal{B}_{r}(\mathbb{R})}$ is a finite signed measure for each $r > 0$ and denote the total variation, positive and negative part of $\nu_{|\mathcal{B}_{r}(\mathbb{R})}$ by $|\nu_{|\mathcal{B}_{r}(\mathbb{R})}|$, $\nu^{+}_{|\mathcal{B}_{r}(\mathbb{R})}$ and $\nu^{-}_{|\mathcal{B}_{r}(\mathbb{R})}$ respectively. Then the \textnormal{total variation} $|\nu|$, the \textnormal{positive part} $\nu^{+}$ and the \textnormal{negative part} $\nu^{-}$ of $\nu$ are defined to be the unique measures on $(\mathbb{R},\mathcal{B}(\mathbb{R}))$ satisfying
	\begin{equation*}
	|\nu|(\{0\})=\nu^{+}(\{0\})=\nu^{-}(\{0\})=0
	\end{equation*}
	\begin{equation*}
	\text{and}\quad|\nu|(A)=|\nu_{|\mathcal{B}_{r}(\mathbb{R})}|,\,\,\nu^{+}(A)=\nu_{|\mathcal{B}_{r}(\mathbb{R})}^{+}(A),\,\,\nu^{-}(A)=\nu_{|\mathcal{B}_{r}(\mathbb{R})}^{-}(A),
	\end{equation*}
	for $A\in\mathcal{B}_{r}(\mathbb{R})$, for some $r>0$.
\end{defn}
As mentioned in \cite{LPS}, $\nu$ is not a a signed measure because it is defined on $\mathcal{B}_{0}(\mathbb{R})$, which is not a $\sigma$-algebra. In the case it is possible to extend the definition of $\nu$ to $\mathcal{B}(\mathbb{R})$ such that $\nu$ is a signed measure then we will identify $\nu$ with its extension to $\mathcal{B}(\mathbb{R})$ and speak of $\nu$ as a signed measure. Moreover, the uniqueness of $|\nu|$, $\nu^{+}$ and $\nu^{-}$ is ensured by an application of the Carath\'{e}odory's extension theorem (see Lemma 2.14 in \cite{Pass}). Further, notice that $\mathcal{B}_{0}(\mathbb{R})= \{B\in\mathcal{B}(\mathbb{R}):0\notin \overline{B} \}\neq \{B\in\mathcal{B}(\mathbb{R}):0\notin B \}$ (see Remark 2.6 in \cite{Pass}).
\begin{defn}[quasi-L\'{e}vy type measure, quasi-L\'{e}vy measure, QID distribution, from \cite{LPS}]
	A \textnormal{quasi-L\'{e}vy type measure} is a function $\nu: \mathcal{B}_{0}(\mathbb{R})\rightarrow \mathbb{R}$ satisfying the
	condition in Definition \ref{def1} and such that its total variation $|\nu|$ satisfies $\int_{\mathbb{R}} (1\wedge x^{2} ) |\nu|(dx) <\infty$. Let $\mu$ be a probability distribution on $\mathbb{R}$. We say that $\mu$ is \textnormal{quasi-infinitely divisible} if its characteristic function has a representation
	\begin{equation*}
	\hat{\mu}(\theta)=\exp\left(i\theta \gamma-\frac{\theta^{2}}{2}a+\int_{\mathbb{R}}\left(e^{i\theta x}-1-i\theta\tau(x)\right)\nu(dx)\right)
	\end{equation*}
	where $a, \gamma \in \mathbb{R}$ and $\nu$ is a quasi-L\'{e}vy type measure. The characteristic triplet $(\gamma,a,\nu)$
	of $\mu$ is unique, and $a$ and $\gamma$ are called the \textnormal{Gaussian variance} and the \textnormal{drift} of $\mu$, respectively. A quasi-L\'{e}vy type measure $\nu$ is called \textnormal{quasi-L\'{e}vy measure}, if additionally there exist a quasi-infinitely divisible distribution $\mu$ and some $a,\gamma\in\mathbb{R}$	such that $(\gamma,a,\nu)$ is the characteristic triplet of $\mu$. We call $\nu$ the quasi-L\'{e}vy measure of $\mu$.
\end{defn}
The above definition extend to the $\mathbb{R}^{d}$ case (for $d>1$) as shown in Remark 2.4 in \cite{LPS}. As pointed out in Example 2.9 of \cite{LPS}, a quasi-L\'{e}vy measure is always a quasi-L\'{e}vy type measure, while the converse is not true. Moreover, we say that a function $f$ is \textit{integrable with respect to quasi-L\'{e}vy type measure} $\nu$ if it is integrable with respect to $|\nu|$. Then, we define:
\begin{equation*}
\int_{B}fd\nu:=\int_{B}fd\nu^{+}-\int_{B}fd\nu^{-},\quad B\in\mathcal{B}(\mathbb{R}).
\end{equation*}
In this work we always keep the same order for the elements in the characteristic triplet: the first element is the drift, the second one is the Gaussian variance, and the third one is the (quasi) L\'{e}vy measure.
\begin{defn}
	[QID random measure] Let $\Lambda$ be a random measure. If $\Lambda(A)$ is a QID random variable, for every $A\in\mathcal{S}$, then we call $\Lambda$ a QID random measure.
\end{defn}
\begin{rem}
	By construction any ID random measure is a QID random measure, but the converse is not always true, namely the set of QID random measures is strictly larger than the set of ID random measures.
\end{rem}
Finally, we remark that the practical example one should have in mind when dealing with random measures is the following: a non-negative random measure $\Lambda$ on $\mathcal{B}_{b}(\mathbb{R})$ (\textit{i.e.}~the set of bounded intervals of $\mathbb{R}$), which has almost surely finite values for any $B\in\mathcal{B}_{b}(\mathbb{R})$. In this example, using the notation of this work, we have that $S=\mathbb{R}$ and ${\mathcal {S}}=\mathcal{B}_{b}(\mathbb{R})$. Then, under certain conditions (\textit{i.e.}~$\Lambda$ is independently scattered) and parametrisations (\textit{i.e.}~$\Lambda([0,t])$) it is possible to associate an additive stochastic process to $\Lambda$. In particular, let $X_{t}\stackrel{a.s.}{=}\Lambda([0,t])$ then $(X_{t})_{t\in[0,\infty)}$ is a non-negative additive process. Observe that $\mathcal{B}_{b}(\mathbb{R})$ is not an algebra because $\mathbb{R}\notin \mathcal{B}_{b}(\mathbb{R})$ and is not a $\sigma$-ring, but a $\delta$-ring, because it is not closed under countable union. From this example, it also appears clear and natural the condition that imposes the existence of an increasing sequence of sets $S_{1},S_{2},\dots \in \mathcal{B}_{b}(\mathbb{R})$ s.t.~$\bigcup _{n\in \mathbb {N} }S_{n}=S$; indeed, think of $S_{n}$ as concentric balls of radii $n$, namely $(-n,n)$. 
\section{Extension of signed bimeasures and existence of relative kernels}\label{Sec-Spectral}
In this section, the necessary and sufficient condition for the extension of signed bimeasures on $\delta$-rings. This result extends the classical results in \cite{Horo} on the extension of signed bimeasures and Theorem 5.18 in \cite{Pass} on the Jordan decomposition of such extension. In particular, in \cite{Horo} the author shows that there exists a signed measure on the product space, under the assumption that the bimeasure is a bimeasure on the Cartesian product of two $\sigma$-algebra and that it satisfies a boundedness conditions. Here, we weaken the first part of the assumption (without strengthening the other one). Indeed, we work with spaces which are not necessarily measurable. 

Moreover, in \cite{Horo} there is no Jordan decomposition of the extended signed measure. Such decomposition has been recently proved in \cite{Pass} (see Theorem 5.18 in \cite{Pass}). Here, we prove the same result but under weaker assumptions.

Theorem \ref{pr-monster} provides a general and complete ``signed" version of the fundamental results at the base of Rajput and Rosinski's work \cite{RajRos} (see Lemma 2.3 and Proposition 2.4 and all the results in their proofs). Theorem \ref{pr-monster} can also be seen as the extension of the construction method of regular conditional probabilities when the probability measure is a signed measure.

Let us now recall the following results proved in \cite{Pass}.
\begin{lem}[Lemma 2.15 in \cite{Pass}]\label{215}
	Let $X$ an arbitrary non-empty set and let $\mathcal{R}$ be a $\delta$-ring. Then, for every $E\in\mathcal{R}$ we have that $\{E\cap B:B\in\mathcal{S}\}$ is a $\sigma$-algebra.
\end{lem}
\begin{lem}[Lemma 2.16 in \cite{Pass}]
	Let $X$ an arbitrary non-empty set and let $\mathcal{R}$ a $\delta$-ring. Let $\mu$ be a (possibly infinite) signed measure on $\mathcal{R}$. Then, there exist two unique measures $\mu^{+}$ and $\mu^{-}$ on $\mathcal{R}$ such that $\mu=\mu^{+}-\mu^{-}$ and that on any $A\in\mathcal{R}$ they are mutually singular.
\end{lem}
\begin{lem}[Lemma 2.18 in \cite{Pass}]
	Let $X$ an arbitrary non-empty set and let $\mathcal{R}$ a $\delta$-ring. Let $\mu$ be a $\sigma$-finite signed measure on $\mathcal{R}$ (namely there exists a sequence $S_{1},S_{2},...\in\mathcal{R}$ s.t.~$X=\cup_{n=1}^{\infty}S_{n}$ and that $-\infty<\mu(S_{n})<\infty$ for every $n\in\mathbb{N}$). Then $\mu^{+}$ and $\mu^{-}$ can be uniquely extended to two $\sigma$-finite measures on $(X,\sigma(\mathcal{R}))$.
\end{lem}
From the above it is possible to see that $|\mu|=\mu^{+}+\mu^{-}$ is the total variation of $\mu$. Moreover, recall that a measurable space $(X,\mathcal{B})$ is a Lusin measurable space if $X$ is measure-theoretically isomorphic to a Borel subset of a compact metric space, and $\mathcal{B}$ is the induced $\sigma$-algebra. Thus, from definition we have that $\mathcal{B}$ is separable and that any standard Borel space is a Lusin measurable space.

\begin{thm}\label{pr-monster}
	Let $(X,\mathcal{B})$ be a Lusin measurable space and let $(T,\mathcal{A})$ be such that $T$ is an arbitrary non-empty set and $\mathcal{A}$ is a $\delta$-ring with the additional condition that there exists an increasing sequence of sets $T_{1},T_{2},\dots \in {\mathcal {A}}$ s.t.~$\bigcup _{n\in \mathbb {N} }T_{n}=T$. Let $Q_{0}(A,B)$ be a (possibly real valued) set function of $A\in\mathcal{A}$, $B\in\mathcal{B}$, satisfying:
	\\ \textnormal{(a)} for every $A\in\mathcal{A}$, $Q_{0}(A,\cdot)$ is a signed measure on $(X,\mathcal{B})$,
	\\ \textnormal{(b)} for every $B\in\mathcal{B}$, $Q_{0}(\cdot, B)$ is a signed measure on $(T,\mathcal{A})$,
	\\ \textnormal{(c)} $\sup\limits_{I_{A}}\sum_{i\in I_{A}}|Q_{0}(A_{i},B_{i})|<\infty$, for every $A\in\mathcal{A}$, where the supremum is taken over all the finite families of disjoints elements of $(\mathcal{A}\cap A)\times\mathcal{B}$.
	\\Let $\nu(A):=\sup\limits_{I_{A}}\sum_{i\in I_{A}}|Q_{0}(A_{i},B_{i})|$, where $A\in\mathcal{A}$. Then, $\nu(\cdot)$ has a unique extension on $(T,\sigma(\mathcal{A}))$. Further, there exist two unique measures $Q^{+}$ and $Q^{-}$ on $\sigma(\mathcal{A})\otimes\mathcal{B}$ s.t.
	\begin{equation}\label{+and-}
	Q^{+}(C)=\int_{T}\int_{X}\textbf{1}_{C}(x,t)q^{+}(t,dx)\nu(dt)\quad\text{and}\quad Q^{-}(C)=\int_{T}\int_{X}\textbf{1}_{C}(x,t)q^{-}(t,dx)\nu(dt),
	\end{equation}
	where $C\in \sigma(\mathcal{A})\otimes\mathcal{B}$. Moreover, there exists a unique finite signed measure on the $\delta$-ring $\bigcup_{D\in\mathcal{A}}(\mathcal{A}\cap D)\otimes\mathcal{B}$ s.t. \begin{equation}\label{formula-in-the-moster}
	Q(A\times B)=Q_{0}(A,B)=\int_{A}q(t,B)\nu(dt),
	\end{equation}
	for every $A\in\mathcal{A}$, $B\in\mathcal{B}$,  where $q:T\times\mathcal{B}\rightarrow[-1,1]$ and $q^{+},q^{-}:T\times\mathcal{B}\rightarrow[0,1]$ fulfil the following conditions:
	\\ \textnormal{(d)} for every $t\in T$, $q(t,\cdot)$ is a signed measure on $\mathcal{B}$,
	\\ \textnormal{(e)} for every $B\in\mathcal{B}$, $q(\cdot,B)$ is $\sigma(\mathcal{A})$-measurable,
	\\ \textnormal{(d)$^{\prime}$} $q^{+}(t,\cdot)$ and $q^{-}(t,\cdot)$ are the Jordan decomposition of $q(t,\cdot)$, \\ \textnormal{(e)$^{\prime}$} $q^{+}(\cdot,B)$ and $q^{-}(\cdot,B)$ are $\sigma(\mathcal{A})$-measurable functions.
	\\ Further, if $q_{1}(\cdot,\cdot)$ is some other function satisfying $(\ref{formula-in-the-moster})$, \textnormal{(d)} and \textnormal{(e)}, then off a set of $\nu$-measure zero, $q_{1}(t,\cdot)=q(t,\cdot)$. Similarly, if $q_{1}^{+}(\cdot,\cdot)$ and $q_{1}^{-}(\cdot,\cdot)$ are some other function satisfying $(\ref{+and-})$, \textnormal{(d)'} and \textnormal{(e)'}, then off a set of $\nu$-measure zero, $q_{1}^{+}(t,\cdot)=q^{+}(t,\cdot)$ and $q_{1}^{-}(t,\cdot)=q^{-}(t,\cdot)$. Finally, condition \textnormal{(c)} is a necessary and sufficient condition for the existence of $Q$.
\end{thm}
\begin{proof}
	Since $\sup\limits_{I_{T_{n}}}\sum_{i\in I_{T_{n}}}|Q_{0}(A_{i},B_{i})|<\infty$ then we have that $\nu(T_{n})<\infty$ for every $n\in\mathbb{N}$. By Lemma \ref{215} we know that $(A,\mathcal{A}\cap A)$ is a $\sigma$-algebra, and by Theorem 4 in \cite{Horo} (see also Theorem 5.17 in \cite{Pass}) we know that $\nu(\cdot)$ is a finite measure on $(A,\mathcal{A}\cap A)$, for every $A\in\mathcal{A}$. Then, $\nu(\cdot)$ is a measure on $\mathcal{A}$ and by the Carath\'{e}odory's extension theorem we know that there exists a unique $\sigma$-finite extension on $\sigma(\mathcal{A})$, which we still denote it by $\nu$. 
	
	Consider the measurable space $(T_{n},\mathcal{A}\cap T_{n})$. By Theorem 5.18 in \cite{Pass} we have that there exist two unique finite measures
	\begin{equation*}
	Q_{n}^{+}(C)=\int_{T_{n}}\int_{X}\textbf{1}_{C}(x,y)q_{n}^{+}(x,dy)\nu(dx)\,\,\,\,\,\textnormal{and} \,\,\,\,\, Q_{n}^{-}(C)=\int_{T_{n}}\int_{X}\textbf{1}_{C}(x,y)q_{n}^{-}(x,dy)\nu(dx),
	\end{equation*}
	and a unique finite signed measure
	\begin{equation*}
	Q_{n}(C)=\int_{T_{n}}\int_{X}\textbf{1}_{C}(x,y)q_{n}(x,dy)\nu(dx),
	\end{equation*}	
	where $C\in (T_{n},\mathcal{A}\cap T_{n})\otimes(X,\mathcal{B})$, $q_{n}^{+}$ and $q_{n}^{-}$ are two sub-Markovian kernels such that for every $x\in T_{n}$ they are the Jordan decomposition of a finite signed measure $q_{n}$. In particular, we have that for very  $A\in \mathcal{A}\cap T_{n}$ and $B\in\mathcal{B}$
	\begin{equation*}
	Q_{0}(A,B)=Q_{n}(A\times B)=\int_{A}q_{n}(x,B)\nu(dx).
	\end{equation*}
	
	Observe that the above holds for any $n\in\mathbb{N}$. Now, we want to concatenate the sequence of obtained $q^{+}_{n}$'s into one measure. For this purpose, let $q^{+}(x,B)=q_{n}^{+}(x,B)$ when $x\in T_{n}\setminus T_{n-1}$ and $b\in\mathcal{B}$. Then, it is possible to see that $q^{+}(x,\cdot)$ is a measure for every $x\in T$. This is because for every $x\in T$ there exists a $n\in\mathbb{N}$ large enough such that $x\in T_{n}\setminus T_{n-1}$ and so $q^{+}(x,\cdot)=q^{+}_{n}(x,\cdot)$, and we know that $q^{+}_{n}(x,\cdot)$ is a measure on $(X,\mathcal{B})$. It is also possible to see that $q^{+}(\cdot,B)$ is a $\sigma(\mathcal{A})$-measurable function, namely that for every $B\in\mathcal{B}$ and every $A\in\mathcal{B}(\mathbb{R})$ we have that $(q^{+}(\cdot,B))^{-1}(A)\in\sigma(\mathcal{A})$. Indeed, consider any $A\in\mathcal{B}(\mathbb{R})$ and $B\in\mathcal{B}$, then
	\begin{equation*}
	(q^{+}(\cdot,B))^{-1}(A)=\{x\in T\,\,|\,\,q^{+}(x, B)\in A \}=\bigcup_{n=1}^{\infty}\{x\in T_{n}\setminus T_{n-1}\,\,|\,\,q_{n}^{+}(x, B)\in A \}
	\end{equation*}
	\begin{equation*}
	=\bigcup_{n=1}^{\infty}\{x\in T_{n}\,\,|\,\,q_{n}^{+}(x, B)\in A \}\setminus T_{n-1}.
	\end{equation*}
	Further, since $q_{n}$ is a $\mathcal{A}\cap T_{n}$-measurable functions and $T_{n-1}\in \mathcal{A}\cap T_{n}$, then $\{x\in T_{n}\,\,|\,\,q_{n}^{+}(x, B)\in A \}\setminus T_{n-1}\in \mathcal{A}\cap T_{n}$ and since $\mathcal{A}\cap T_{n}\subset\sigma(\mathcal{A})$ then $\{x\in T_{n}\,\,|\,\,q_{n}^{+}(x, B)\in A \}\setminus T_{n-1}\in \sigma(\mathcal{A})$. Therefore, since $\sigma$-algebras are closed under countable unions we have that $\bigcup_{n=1}^{\infty}\{x\in T_{n}\,\,|\,\,q_{n}^{+}(x, B)\in A \}\setminus T_{n-1}\in\sigma(\mathcal{A})$ and so that $q^{+}(\cdot,B)$ is a $\sigma(\mathcal{A})$-measurable function, for every $B\in\mathcal{B}$. 
	
	Similarly we can define $q^{-}$ and by applying the same arguments we have that $q^{-}_{n}(x,\cdot)$ is a measure on $(X,\mathcal{B})$ and $q^{-}(\cdot,B)$ is a $\sigma(\mathcal{A})$-measurable function. Then, it is possible to define two (possibly infinite) measures $Q^{+}$ and $Q^{-}$ on $\sigma(\mathcal{A})\otimes\mathcal{B}$ by
	\begin{equation*}
	Q^{+}(C)=\int_{T}\int_{X}\textbf{1}_{C}(x,y)q^{+}(x,dy)\nu(dx)\,\,\,\,\,\textnormal{and} \,\,\,\,\, Q^{-}(C)=\int_{T}\int_{X}\textbf{1}_{C}(x,y)q^{-}(x,dy)\nu(dx),
	\end{equation*}
	where $C\in \sigma(\mathcal{A})\otimes\mathcal{B}$.
	
	Notice that since $q^{+}_{n}(x,B)\leq 1$ and $q^{-}_{n}(x,B)\leq 1$ for every $x\in T_{n}$ and $B\in\mathcal{B}$ and since this holds for every $n\in\mathbb{N}$, then $q^{+}(x,B)\leq 1$ and $q^{-}(x,B)\leq 1$ for every $x\in T$ and $B\in\mathcal{B}$. In other words, $q^{+}$ and $q^{-}$ are sub-Markovian kernels. Then we can define $q$ to be $q(x,B)=q^{+}(x,B)-q^{-}(x,B)$ and notice that for every $x\in T$ $q^{+}(x,\cdot)$ and $q^{-}(x,\cdot)$ are the Jordan decomposition of $q(x,\cdot)$, and that for every $n\in\mathbb{N}$ we have $q(x,B)=q_{n}(x,B)$ for every $x\in T_{n}$ and $B\in\mathcal{B}$. 
	
	Therefore, by putting together the results obtained so far we have, for every $A\in \mathcal{A}$ and $B\in\mathcal{B}$, that $Q_{0}(A,B)<\infty$ and in particular that
	\begin{equation*}
	Q_{0}(A,B)=\sum_{n=1}^{\infty}Q_{0}(A\cap T_{n}\setminus T_{n-1},B)=\sum_{n=1}^{\infty}\int_{A\cap T_{n}\setminus T_{n-1}}q_{n}(x,B)\nu(dx)
	\end{equation*}
	\begin{equation*}
	=\sum_{n=1}^{\infty}\int_{A\cap T_{n}\setminus T_{n-1}}q(x,B)\nu(dx)=\int_{A}q(x,B)\nu(dx)=Q^{+}(A\times B)-Q^{-}(A\times B).
	\end{equation*}
	
	Now, observe that it is possible to define a real valued countably additive set function $Q$ on $\bigcup_{D\in\mathcal{A}}(\mathcal{A}\cap D)\otimes\mathcal{B}$ by setting $Q(C)=Q^{+}(C)-Q^{-}(C)$, namely
	\begin{equation*}
	Q(C)=\int_{T}\int_{X}\textbf{1}_{C}(x,y)q(x,dy)\nu(dx),\quad C\in \bigcup_{D\in\mathcal{A}}(\mathcal{A}\cap D)\otimes\mathcal{B}.
	\end{equation*}
	Indeed, consider any $D\in\mathcal{A}$. For every $A\in \mathcal{A}\cap D$ and $B\in\mathcal{B}$ we have that
	\begin{equation*}
	Q_{0}(A,B)=\int_{A}q(x,B)\nu(dx)=Q(A\times B).
	\end{equation*}
	Therefore, following Thoerem 5.18 in \cite{Pass} $Q$ is the unique finite signed measure on $(D\cap\mathcal{A})\otimes\mathcal{B}$ s.t.
	\begin{equation*}
	Q(C)=\int_{T}\int_{X}\textbf{1}_{C}(x,y)q(x,dy)\nu(dx),\quad C\in (D\cap\mathcal{A})\otimes\mathcal{B}.
	\end{equation*}
	Indeed, it is possible to see this also by the following arguments. For every $D\in\mathcal{A}$ consider the measurable space $(D,\mathcal{A}\cap D)$ and by applying Theorem 5.18 in \cite{Pass} we obtain $Q_{D}$, $Q_{D}^{+}$, $Q_{D}^{-}$, $q^{+}_{D}$, $q^{-}_{D}$ and $q_{D}$ (as we have done before when we obtained $Q_{n}$, $Q_{n}^{+}$, $Q_{n}^{-}$, $q^{+}_{n}$, $q^{-}_{n}$ and $q_{n}$). Then, we would have that for every $A\in \mathcal{A}\cap D$ and $B\in\mathcal{B}$
	\begin{equation*}
	\int_{A}q_{D}(x,B)\nu(dx)=Q_{0}(A,B)=\int_{A}q(x,B)\nu(dx)
	\end{equation*}
	and so $q_{D}(x,\cdot)=q(x,\cdot)$ off a set of $\nu$-measure zero on $D$. Thus, for every $C\in (D\cap\mathcal{A})\otimes\mathcal{B}$
	\begin{equation*}
	Q_{D}(C)=\int_{T}\int_{X}\textbf{1}_{C}(x,y)q_{D}(x,dy)\nu(dx)=\int_{T}\int_{X}\textbf{1}_{C}(x,y)q(x,dy)\nu(dx)=Q(C).
	\end{equation*}
	
	Finally, we prove uniqueness. If $q_{1}(\cdot,\cdot)$ is some other function satisfying $(\ref{formula-in-the-moster})$, \textnormal{(d)} and \textnormal{(e)}, then, off a set of $\nu$-measure zero, $q_{1}(t,\cdot)=q(t,\cdot)$ when $t\in T_{n}$. Thus,
	\begin{equation*}
	\int_{T_{n}}q(x,B)\nu(dx)-\int_{T_{n}}q_{1}(x,B)\nu(dx)=0.
	\end{equation*}
	Since this holds for every $n\in\mathbb{N}$, we have
	\begin{equation*}
	\int_{T}q(x,B)-q_{1}(x,B)\nu(dx)=\lim\limits_{n\rightarrow\infty}\int_{T_{n}}q(x,B)-q_{1}(x,B)\nu(dx)=0,
	\end{equation*}
	for every $B\in\mathcal{B}$. Hence, off a set of $\nu$-measure zero, $q_{1}(t,\cdot)=q(t,\cdot)$, thus we get the uniqueness of $Q$. Now, from this and from the uniqueness of the Jordan decomposition we deduce that, off a set of $\nu$-measure zero, $q^{+}_{1}(t,\cdot)=q^{+}(t,\cdot)$ and $q^{-}_{1}(t,\cdot)=q^{-}(t,\cdot)$, whence we obtain the uniqueness of $Q^{+}$ and $Q^{-}$.
	
	The condition (c) is necessary because $Q$ is a finite signed measure on $\bigcup_{D\in\mathcal{A}}(\mathcal{A}\cap D)\otimes\mathcal{B}$, hence for every $A\in\mathcal{A}$ the total variation of $Q$, where $Q$ considered as a signed measure on the $\sigma$-algebra $(\mathcal{A}\cap A)\otimes\mathcal{B}$ is larger than $\nu(A)$.
\end{proof}
\begin{rem}
	\textnormal{First, notice that $\bigcup_{D\in\mathcal{A}}(\mathcal{A}\cap D)\otimes\mathcal{B}=\{C\in(\mathcal{A}\cap D)\otimes\mathcal{B},$ for some $D\in\mathcal{A}\}$. Second, in many situations $\bigcup_{D\in\mathcal{A}}(\mathcal{A}\cap D)\otimes\mathcal{B}=\bigcup_{n\in\mathbb{N}}(\mathcal{A}\cap T_{n})\otimes\mathcal{B}$. For example, this is the case of measures that takes finite values on bounded set of $\mathbb{R}$. In that case $T_{n}$'s are the concentric balls around zero and radius $n$, then for every $A\in\mathcal{A}$ there exists an $n\in\mathbb{N}$ large enough such that $A\subset T_{n}$ and so $(A,\mathcal{A}\cap A)$ is contained in $(T_{n},\mathcal{A}\cap T_{n})$. In general, this is the case when $\mathcal{A}$ is a localising ring of a measurable space, as it is the case in \cite{Kallenberg2} (see page 15 and 19 in \cite{Kallenberg2}). Third, notice that while
	\begin{equation*}
	\int_{T}q(x,B)-q_{1}(x,B)\nu(dx)
	\end{equation*}
	is well defined and finite, the objects $\int_{T}q(x,B)\nu(dx)$ and $\int_{T}q_{1}(x,B)\nu(dx)$ are not well-defined.}
\end{rem}
\begin{rem}
		\textnormal{Since $(T,\sigma(\mathcal{A}))$ and $(X,\mathcal{B})$ are two measurable spaces, we remark that $q^{+}$ and $q^{-}$ are sub-Markovian kernels from $T$ to $X$ (see page 16 in \cite{Kallenberg2}).}
\end{rem}
\begin{rem}
	Theorem \ref{pr-monster} extends the construction method of regular conditional probabilities, see Proposition 2.4 in \cite{RajRos} for a general result, in two directions. First, $\mathcal{A}$ in the space $(T,\mathcal{A})$ is now a $\delta$-ring with the stated additional condition, which is more general than considering $(T,\mathcal{A})$ to be just a measurable space as done in Proposition 2.4 in \cite{RajRos}. Second, $Q_0$ is now a signed bimeasure (precisely a countably additive real valued bi-set function) which is more general than a bimeasure as considered in Proposition 2.4 in \cite{RajRos}.
\end{rem}
\begin{rem}
	In \cite{Karni}, the authors obtained the result on the extension of signed bimeasures obtained in \cite{Horo} for $\sigma$-algebras of general measurable spaces, namely without the condition that one of them needs to be a Lusin measurable space. It might be possible to obtain such result also in our setting. However, this result, as the one in \cite{Karni}, will necessary lack the results on Markov kernels which is indeed of fundamental importance for our purpose and in general in probability theory and stochastic analysis, e.g.~random measures (\cite{Kallenberg2}), stochastic processes and stochastic integration (\cite{RajRos,Pass}), regular conditional probabilities, etc... 
\end{rem}
\section{Stochastic integration based on random measures}\label{Sec-Integration}
In this section we provide a complete general theory for stochastic integrals based on QID random measures. Therefore, we extend the main results of \cite{Pass} on QID random measures and all the results in Chapter II and most of the results in Chapter III in \cite{RajRos}.

First, let us introduce the topological setting of this section. We denote by $S$ an arbitrary non-empty set and by $\mathcal{S}$ a $\delta$-ring with the additional condition that there exists an increasing sequence of sets $S_{1},S_{2},\dots \in {\mathcal {S}}$ s.t.~$\bigcup _{n\in \mathbb {N} }S_{n}=S$. In this framework $S$ does not need to belong to $\mathcal{S}$ (thus $\mathcal{S}$ is not necessarily an algebra) and arbitrary subsets of $S$ do not need to satisfy the condition $\bigcup_{n\in \mathbb {N} }A_{n}\in\mathcal{S}$ (thus $\mathcal{S}$ is not necessarily a $\sigma$-ring).

The following result extends Lemma 3.1 in \cite{Pass}. For this result, we use the following notation. Let $\nu_{0}:\mathcal{S}\mapsto\mathbb{R}$ be a signed measure, $\nu_{1}:\mathcal{S}\mapsto[0,\infty)$ be a measure, $F_{A}(\cdot)$ be a quasi-L\'{e}vy type measure for every $A\in\mathcal{S}$, and $F_{\cdot}(B)$ be a finite signed measure for every $B\in\mathcal{B}_{0}(\mathbb{R})$. Observe that such objects are typical for ID and QID random measures, see Section II in \cite{RajRos} and Section 3 and 4 in \cite{Pass}. Further, we define for every $A\in\mathcal{S}$ and $B\in\mathcal{B}(\mathbb{R})$
\begin{equation*}
J(A,B):=\int_{B}(1\wedge x^{2})F_{A}(dx).
\end{equation*}

Consider the following assumption:
\begin{equation}\label{final-assumption}
\sup_{I_{A}}\sum_{i\in I_{A}}|J(A_{i},B_{i})|<\infty,\quad\forall A\in\mathcal{S},
\end{equation}
where the supremum is taken over all the finite families of disjoints elements of $(\mathcal{S}\cap A)\times\mathcal{B}(\mathbb{R})$. In other words, the supremum is taken over all the finite families of the form $(A_{i},B_{i})_{i\in I_{A}}$, where $I_{A}$ is finite, $A_{i}\in\mathcal{S}\cap A$, $B_{i}\in\mathcal{B}(\mathbb{R})$, and $(A_{i}\times B_{i})\cap (A_{j}\times B_{j})=\emptyset$ for every $i,j\in I_{A}$ with $i\neq j$. Since $\int_{\mathbb{R}}(1\wedge x^{2})|F_{A}|(dx)<\infty$, then $J(\cdot,B)$ is a finite signed measure on $\mathcal{S}$ and $J(A,\cdot)$ is a finite signed measure on $\mathcal{B}(\mathbb{R})$. Further, as in \cite{Pass} and \cite{RajRos}, let
\begin{equation*}
\tau(x):=\begin{cases}
x\,\,&\textnormal{if }\,\,\,|x|\leq 1,\\ \frac{x}{|x|} \,\,&\textnormal{if }\,\,\,|x|> 1.
\end{cases}
\end{equation*}
\begin{lem}\label{l1}
Let $\nu_{0}$, $\nu_{1}$, $F$ be as above and let $F$ satisfy assumption (\ref{final-assumption}). If the triplet $(\nu_{0}(A),\nu_{1}(A),F_{A})$ is the characteristic triplet of a QID random variable, $\forall A\in\mathcal{S}$. Then there exists a unique (in the sense of finite-dimensional distributions) QID random measure $\Lambda$ such that, for every $ A\in\mathcal{S}$,
	\begin{equation}\label{cf}
	\hat{\mathcal{L}}(\Lambda(A))(\theta)=\exp\left(i\theta \nu_{0}(A)-\frac{\theta^{2}}{2}\nu_{1}(A)+\int_{\mathbb{R}}e^{i\theta x}-1-i\theta\tau(x)F_{A}(dx)\right).
	\end{equation}
	Moreover, $\int_{B} (1\wedge x^{2} ) F_{A}(dx)$ is a finite signed bimeasure on $\mathcal{S}\times\mathcal{B}(\mathbb{R})$.
\end{lem}
\begin{proof}
The existence of a finitely additive random measure $\Lambda=\{\Lambda(A):A\in\mathcal{S}\}$ follows by a standard application of Kolmogorov extension theorem using the finite additivity of $\nu_{0}(\cdot)$, $\nu_{1}(\cdot)$, and of $\int_{\mathbb{R}}e^{i\theta x}-1-i\theta\tau(x)F_{\cdot}(dx)$, $\forall\theta\in\mathbb{R}$. Indeed, for the latter we have that for $A_{1},A_{2}\in\mathcal{S}$ with $A_{1}\cap A_{2}=\emptyset$
\begin{equation*}
\int_{\mathbb{R}}e^{i\theta x}-1-i\theta\tau(x)F_{A_{1}\cup A_{2}}(dx)
=\int_{|x|\geq\epsilon}e^{i\theta x}-1-i\theta\tau(x)F_{A_{1}\cup A_{2}}(dx)+\int_{|x|<\epsilon}e^{i\theta x}-1-i\theta\tau(x)F_{A_{1}\cup A_{2}}(dx)
\end{equation*}
\begin{equation*}
=\int_{|x|\geq\epsilon}e^{i\theta x}-1-i\theta\tau(x)F_{A_{1}}(dx)+\int_{|x|\geq\epsilon}e^{i\theta x}-1-i\theta\tau(x)F_{A_{2}}(dx)
+\int_{|x|<\epsilon}e^{i\theta x}-1-i\theta\tau(x)F_{A_{1}\cup A_{2}}(dx)
\end{equation*}
\begin{equation*}
\stackrel{\epsilon\to0}{\to}\int_{\mathbb{R}}e^{i\theta x}-1-i\theta\tau(x)F_{A_{1}}(dx)+\int_{\mathbb{R}}e^{i\theta x}-1-i\theta\tau(x)F_{ A_{2}}(dx).
\end{equation*}
To prove that $\Lambda$ is countable additive let $A_{n}\downarrow\emptyset$ with $\{A_{n}\}\subset\mathcal{S}$. Then, by definition of (signed) measures, $\nu_{0}(A_{n})\rightarrow0$ and $\nu_{1}(A_{n})\rightarrow0$. In the following we prove that $\int_{\mathbb{R}} (1\wedge x^{2} ) |F_{A_{n}}|(dx)\to0$ as $n\to\infty$. By the L\'{e}vy continuity theorem, this will give us $\Lambda(A_{n})\stackrel{p}{\to}0$, namely the countable additivity of $\Lambda$.

Let $\epsilon>0$ and let $J_{\epsilon}(A,B):=\int_{\mathbb{R}} (1\wedge x^{2} ) F_{A_{n}}(dx)$ for every $A\in\mathcal{S}$ and $B\in\mathcal{B}_{\epsilon}(\mathbb{R})$ (recall that $\mathcal{B}_{\epsilon}(\mathbb{R})=\{B\in\mathcal{B}(\mathbb{R}):B\cap(-\epsilon,\epsilon)=\emptyset\}$). By the properties of $F$ we have that $J_{\epsilon}$ is a finite signed bimeasure on $(A_{1},\mathcal{S}\cap A_{1})\times \mathcal{B}_{\epsilon}(\mathbb{R})$. Notice that we have taken $A_{1}$ because $A_{n}\in (A_{1},\mathcal{S}\cap A_{1})$ for every $n\in\mathbb{N}$. Thus, by Theorem 5.18 in \cite{Pass} we have that there exists a unique finite signed measure on $(A_{1},\mathcal{S}\cap A_{1})\otimes \mathcal{B}_{\epsilon}(\mathbb{R})$, call it $Q_{\epsilon}$. Observe that its total variation, which we denote by $|Q_{\epsilon}|$, is a well-defined finite measure on $(A_{1},\mathcal{S}\cap A_{1})\otimes \mathcal{B}_{\epsilon}(\mathbb{R})$. Now, let $E^{+}_{A}$ and $E^{-}_{A}$ be the Hahn decomposition of $\mathbb{R}$ under the signed measure $F_{A}(\cdot)$, for every $A\in\mathcal{S}$. Observe that, for every $n\in\mathbb{N}$, we have
\begin{equation*}
|Q_{\epsilon}|(A_{n}\times(-\infty,-\epsilon]\cup[\epsilon,\infty))\geq\sup\limits_{I_{A_{n}}}\sum_{i\in I_{A_{n}}}|J_{\epsilon}(A_{i},B_{i})|
\end{equation*}
\begin{equation*}
\geq \int_{E_{A_{n}}^{+}\cap |x|\geq\epsilon}(1\wedge x^{2})F_{A_{n}}(dx)-\int_{E_{A_{n}}^{-}\cap |x|\geq\epsilon}(1\wedge x^{2})F_{A_{n}}(dx)= \int_{ |x|\geq\epsilon}(1\wedge x^{2})|F_{A_{n}}|(dx).
\end{equation*}
Since $A_{n}\to\emptyset$, then $|Q_{\epsilon}|(A_{n}\times(-\infty,-\epsilon]\cup[\epsilon,\infty))\to0$ as $n\to\infty$. Thus, $\int_{ |x|\geq\epsilon}(1\wedge x^{2})|F_{A_{n}}|(dx)\to0$ as $n\to\infty$, for every $\epsilon>0$.

By (\ref{final-assumption}) we have that 
\begin{equation*}
\infty>\sup\limits_{I_{A_{1}}}\sum_{i\in I_{A_{1}}}|J(A_{i},B_{i})|\geq\sup\limits_{n\in\mathbb{N}}\int_{ |x|<\epsilon}(1\wedge x^{2})|F_{A_{n}}|(dx).
\end{equation*}
Since $\int_{ |x|<\epsilon}(1\wedge x^{2})|F_{A_{n}}|(dx)$ is a decreasing function of $\epsilon$, for every $n\in\mathbb{N}$, we have that $\sup\limits_{n\in\mathbb{N}}\int_{ |x|<\epsilon}(1\wedge x^{2})|F_{A_{n}}|(dx)$ is a bounded and decreasing function of $\epsilon$. Hence, $\sup\limits_{n\in\mathbb{N}}\int_{ |x|<\epsilon}(1\wedge x^{2})|F_{A_{n}}|(dx)\to0$ as $\epsilon\to0$. Therefore,
\begin{equation*}
\lim\limits_{n\to\infty} \int_{\mathbb{R}} (1\wedge x^{2} ) |F_{A_{n}}|(dx)\leq\lim\limits_{n\to\infty} \int_{|x|\geq\epsilon} (1\wedge x^{2} ) |F_{A_{n}}|(dx)+\int_{|x|<\epsilon} (1\wedge x^{2} ) |F_{A_{n}}|(dx)
\end{equation*}
\begin{equation*}
\leq \sup\limits_{n\in\mathbb{N}}\int_{ |x|<\epsilon}(1\wedge x^{2})|F_{A_{n}}|(dx)
\end{equation*}
which goes to zero as $\epsilon\to0$. 

Finally, $B\mapsto\int_{B} (1\wedge x^{2} ) F_{A}(dx)$ is a finite signed measure on $\mathcal{B}(\mathbb{R})$, for every $A\in\mathcal{S}$. Further, since $A\mapsto\int_{B} (1\wedge x^{2} ) F_{A}(dx)$ is finitely additive for every $B\in\mathcal{B}(\mathbb{R})$, and since $\int_{B} (1\wedge x^{2} ) |F_{A_{n}}|(dx)\to 0$ as $n\to\infty$ for every $A_{n}\downarrow\emptyset$, then $\int_{B} (1\wedge x^{2} ) F_{A}(dx)$ is a finite signed measure on $\mathcal{S}$, for every $B\in\mathcal{B}(\mathbb{R})$.
\end{proof}
Notice that in Lemma \ref{l1} we have to mention the sentence ``If $(\nu_{0}(A),\nu_{1}(A),F_{A})$ is the characteristic triplet of a QID random variable, $\forall A\in\mathcal{S}$", because, while for every QID distribution there exists a (unique) characteristic triplet, not every characteristic triplet gives rise to a (QID) distribution. Moreover, the reason why Lemma \ref{l1} is an extension of Lemma 3.1 in \cite{Pass} is because the assumption on $F$ in the former is weaker than the one in the latter (see Lemma \ref{l3}).

In the following result, we provide a trivial generalisation of Lemma 3.4 in \cite{Pass}.
\begin{lem}
Let $\Lambda$ be random measure. Denote by $(\nu_{0}(A),\nu_{1}(A),F_{A})$ the characteristic triplet of $\Lambda(A)$, for every $A\in\mathcal{S}$. Assume that (\ref{final-assumption}) hold and that the following hold: for every $\{A_{n}\}\subset\mathcal{S}$ with $A_{n}\downarrow\emptyset$,
\begin{equation}\label{final-assumption-2}
\int_{\mathbb{R}}e^{i\theta x}-1-i\theta\tau(x)F_{A_{n}}(dx)\to0, \,\,\,\,\forall\theta\in\mathbb{R}\quad\Rightarrow\quad F_{A_{n}}(B)\to0,\,\,\,\,\forall B\in\mathcal{B}_{0}(\mathbb{R}).
\end{equation}
Then, $\nu_{0}$, $\nu_{1}$, $F$ are as in Lemma \ref{l1}, namely $\nu_{0}:\mathcal{S}\mapsto\mathbb{R}$ is a signed measure, $\nu_{1}:\mathcal{S}\mapsto[0,\infty)$ is a measure, $F_{A}(\cdot)$ is a quasi-L\'{e}vy measure for every $A\in\mathcal{S}$, and $F_{\cdot}(B)$ is a signed measure for every $B\in\mathcal{B}_{0}(\mathbb{R})$ and such that $F$ satisfies (\ref{final-assumption}).
\end{lem}
\begin{proof}
First, since $\Lambda$ is a QID random measure, it follows that $F_{A}$ is a quasi-L\'{e}vy measure on $\mathbb{R}$, for every $A\in\mathcal{S}$. second, let $\{A_{k}\}_{k=1}^{n}$, $n\in\mathbb{N}$, be pairwise disjoint sets in $\mathcal{S}$. By the uniqueness of the L\'{e}vy-Khintchine representation of a quasi-ID distribution, it follows, using $\hat{\mathcal{L}}(\Lambda(\bigcup_{k=1}^{n}A_{k}))=\prod_{k=1}^{n}\hat{\mathcal{L}}(\Lambda(A_{k}))$, that all three set functions $\nu_{0}$, $\nu_{1}$ and $F_{\cdot}(B)$ (for every fixed $B\in\mathcal{B}_{0}(\mathbb{R})$) are finitely additive. Let now $\{A_{n}\}\subset\mathcal{S}$, $A_{n}\searrow\emptyset$.
Since $\Lambda(A_{n})\stackrel{p}{\rightarrow}0$ we have that $\nu_{0}(A_{n})\rightarrow0$, $\nu_{1}(A_{n})\rightarrow0$ and $\int_{\mathbb{R}}e^{i\theta x}-1-i\theta\tau(x)F_{A_{n}}(dx)\to0$, $\forall\theta\in\mathbb{R}$. Thus, by (\ref{final-assumption-2}) we obtain the stated result.
\end{proof}
\begin{rem}
	\textnormal{If we restrict to QID random measure satisfying both (\ref{final-assumption}) and (\ref{final-assumption-2}), then from the above lemmas we conclude that there is a one ot one correspondence between a QID random measure satisfying (\ref{final-assumption}) and (\ref{final-assumption-2}), and a triplet composed by a finite signed measure, a finite measure, and a bi-set function, which is a quasi-L\'{e}vy measure for fixed $A\in\mathcal{S}$, is a finite signed measure for fixed $B\in\mathcal{B}_{0}(\mathbb{R})$, and satisfies (\ref{final-assumption}). We would like to have a one to one correspondence without condition (\ref{final-assumption-2}), but this appears to be impossible.}
\end{rem}

We are now ready to present the results on QID stochastic integrals. For the rest of the section we assume the conditions of Lemma \ref{l1}, namely we let: $\nu_{0}$ be a finite signed measure on $\mathcal{S}$, $\nu_{1}$ be a measure on $\mathcal{S}$, $F_{A}(\cdot)$ be a quasi-L\'{e}vy type measure for every $A\in\mathcal{S}$, $F_{\cdot}(B)$ be a finite signed measure for every $B\in\mathcal{B}_{0}(\mathbb{R})$, and condition (\ref{final-assumption}) be satisfied. 

Define the set function $\nu(A):\mathcal{S}\mapsto [0,\infty)$ as
\begin{equation}\label{nu}
\nu(A):=\sup_{I_{A}}\sum_{i\in I_{A}}|J(A_{i},B_{i})|.
\end{equation}
Notice that $\nu(S_{n})<\infty$ and that $\nu$ is a measure on $(S_{n},\mathcal{S}\cap S_{n})$. Then, by the Carath\'{e}odory's extension theorem $\nu$ extends to a $\sigma$-finite measure on $(S,\sigma(\mathcal{S}))$ (see also Theorem \ref{pr-monster}). To have a better idea of what kind of object $\nu$ is, compare it with the definition of total variation of a signed measure (see (\ref{def-totalvariation})).

Let $E^{+}_{A}$ and $E^{-}_{A}$ be the Hahn decomposition of $\mathbb{R}$ under the signed measure $F_{A}$. Observe that
\begin{equation}\label{xi}
\nu(A)=\sup\limits_{I_{A}}\sum_{i\in I_{A}}|J(A_{i},B_{i})|\geq \int_{E_{A}^{+}}(1\wedge x^{2})F_{A}(dx)-\int_{E_{A}^{-}}(1\wedge x^{2})F_{A}(dx)= \int_{\mathbb{R}}(1\wedge x^{2})|F_{A}|(dx).
\end{equation}
Therefore, since $\nu(A)$ is finite by assumption we have that $\int_{\mathbb{R}}(1\wedge x^{2})|F_{A}|(dx)<\infty$. 

We show now that the assumption of this setting, namely (\ref{final-assumption}), is weaker than the ones presented in Section 5 in \cite{Pass}. Indeed, in the rest of this section we both unify and generalise the results on QID stochastic integrals in \cite{Pass}. In Section 5 in \cite{Pass}, the authors investigate three different sets of assumptions under which the results on QID stochastic integrals are obtained.

In Subsection 5.1 in \cite{Pass}, they assume that the QID random measure is ``generated'' by two ID random measures. In other words, they assume that there exist two ID random measures $\Lambda_{G}$ and $\Lambda_{M}$ s.t.~for every $A\in\mathcal{S}$, $\Lambda(A)+\Lambda_{M}(A)\stackrel{d}{=}\Lambda_{G}(A)$ and $\Lambda_{M}(A)$ independent of $\Lambda(A)$. This case resembles the definition of QID distributions transferred to random measures. In this case, $F$ is given by $F_{A}(B)=G_{A}(B)-M_{A}(B)$ for every $A\in\mathcal{S}$ and $B\in\mathcal{B}_{0}(\mathbb{R})$, where $G$ and $M$ are the L\'{e}vy measure of $\Lambda_{G}$ and $\Lambda_{M}$. In Subsection 5.2, they assume that $\mathcal{S}$ is a $\sigma$-algebra and that $F$ is a finite signed bimeasure. In Subsection 5.3, they assume that $\mathcal{S}$ is a $\sigma$-algebra and that $\int_{B}(1\wedge x^{2})F_{A}(dx)$ is a finite signed bimeasure and that (\ref{final-assumption}) holds. It is possible to see that the assumptions of Subsection 5.3 are strictly weaker than the ones in Subsection 5.2. 

The assumptions of Subsection 5.3 in \cite{Pass} are stricter than our assumptions. Indeed, assuming that $\mathcal{S}$ is a $\sigma$-algebra is more restrictive than assuming that $\mathcal{S}$ is a $\delta$-ring with the additional condition that there exists an increasing sequence of sets $S_{1},S_{2},\dots \in {\mathcal {S}}$ s.t.~$\bigcup _{n\in \mathbb {N} }S_{n}=S$, which is the present setting. Moreover, in Lemma \ref{l1} we are able to show that (\ref{final-assumption}) is enough to ensure that $\int_{B}(1\wedge x^{2})F_{A}(dx)$ is a finite signed bimeasure. 

Concerning the assumptions in Section 5.1 in \cite{Pass}, the following result shows that they are stricter than our assumption.
\begin{lem}\label{l3}
	Let $G$ and $M$ be defined as follow: $G_{A}(\cdot)$ is a L\'{e}vy measure for every $A\in\mathcal{S}$ and $G_{\cdot}(B)$ is a measure for every $B\in\mathcal{B}_{0}(\mathbb{R})$ -- and similarly for $M$. Let $F_{A}(B)=G_{A}(B)-M_{A}(B)$ for every $A\in\mathcal{S}$ and $B\in\mathcal{B}_{0}(\mathbb{R})$. Define $\nu$ as in (\ref{nu}). Then,
	\begin{equation*}
	\int_{\mathbb{R}}(1\wedge x^{2})G_{A}(dx)+\int_{\mathbb{R}}(1\wedge x^{2})M_{A}(dx)<\infty,\,\,\,\, \forall A\in\mathcal{S}\quad\Rightarrow\quad \textnormal{Assumption (\ref{final-assumption})}.
	\end{equation*}	
\end{lem} 
\begin{proof}
	Notice that $\nu(A)<\infty$, $\forall A\in\mathcal{S}$, is equivalent to Assumption (\ref{final-assumption}) and that
	\begin{equation*}
	\int_{B}(1\wedge x^{2})F_{A}(dx)=\int_{B}(1\wedge x^{2})G_{A}(dx)-\int_{B}(1\wedge x^{2})M_{A}(dx).
	\end{equation*}
	Thus, for every $A\in\mathcal{S}$, we have that
	\begin{equation*}
	\nu(A)=\sup\limits_{I_{A}}\sum_{i\in I_{A}}|J(A_{i},B_{i})|\leq\sup\limits_{I_{A}}\sum_{i\in I_{A}}\int_{B_{i}}(1\wedge x^{2})G_{A_{i}}(dx)+\int_{B_{i}}(1\wedge x^{2})M_{A_{i}}(dx).
	\end{equation*}
	In the following, we prove that for every family $(A_{i},B_{i})_{i\in I_{A}}$ we have
	\begin{equation*}
	\sum_{i\in I_{A}}\int_{B_{i}}(1\wedge x^{2})G_{A_{i}}(dx)\leq \int_{\mathbb{R}}(1\wedge x^{2})G_{A}(dx).
	\end{equation*}
	If the $A_{i}$'s are all disjoints, then the $B_{i}$'s could take any values. In particular, by the (finite) additivity of $G$ we obtain that
	\begin{equation*}
	\sum_{i\in I_{A}}\int_{B_{i}}(1\wedge x^{2})G_{A_{i}}(dx)\leq \sum_{i\in I_{A}}\int_{\mathbb{R}}(1\wedge x^{2})G_{A_{i}}(dx)= 
		\int_{\mathbb{R}}(1\wedge x^{2})G_{\cup_{i\in I} A_{i}}(dx)= \int_{\mathbb{R}}(1\wedge x^{2})G_{A}(dx).
	\end{equation*}
	Thus, it remains to investigate the case where the $A_{i}$'s have at least one intersection. Let $(A_{i})_{i\in I}$ be any finite family of sets in $\mathcal{S}\cap A$. It is possible to find a finite set of disjoints elements in $\mathcal{S}\cap A$, denote it $(A'_{i})_{i\in J_{A}}$, such that $\bigcup_{i\in I_{A}}A_{i}=\bigcup_{i\in J_{A}}A'_{i}$. Hence, each $A'_{i}$ is a subset of one or more of the $A_{i}$'s. Therefore, the corresponding $B_{i}$ of the $A_{i}$'s, whose intersection is $A'_{i}$, cannot have intersections, because the rectangles $(A_{i},B_{i})$'s must be disjoint. This implies that for each $A'_{i}$ the union of the corresponding $B_{i}$'s is a subset of $\mathbb{R}$. Hence, we have
	\begin{equation*}
	\sum_{i\in I_{A}}\int_{B_{i}}(1\wedge x^{2})G_{A_{i}}(dx)\leq \sum_{i\in J_{A}}\int_{\mathbb{R}}(1\wedge x^{2})G_{A'_{i}}(dx)= \int_{\mathbb{R}}(1\wedge x^{2})G_{\cup_{i\in J_{A}} A'_{i}}(dx)
	\end{equation*}
	\begin{equation*}
	= \int_{\mathbb{R}}(1\wedge x^{2})G_{\cup_{i\in I_{A}} A_{i}}(dx)= \int_{\mathbb{R}}(1\wedge x^{2})G_{A}(dx).
	\end{equation*}
	Since the same arguments hold for $M$, we obtain the stated result.
\end{proof}
\begin{rem}
	\textnormal{This remark contains one of the most important take-home message of this paper. From the proof of Lemma \ref{l3} and from (\ref{xi}), it is possible to see that for $G$ we have that for every $A\in\mathcal{S}$
	\begin{equation*}
	\sup\limits_{I_{A}}\sum_{i\in I_{A}}\int_{B_{i}}(1\wedge x^{2})G_{A_{i}}(dx)=\int_{\mathbb{R}}(1\wedge x^{2})G_{A}(dx).
	\end{equation*}
	In the present setting for $G$, as well as for $M$ and for the L\'{e}vy measure of any ID random measure in Rajput and Rosinski \cite{RajRos}, we know that $\int_{\mathbb{R}}(1\wedge x^{2})G_{A}(dx)<\infty$, because $G$ is L\'{e}vy measure. This implies that assumption (\ref{final-assumption}) in these cases is \textit{always} satisfied. Thus, assumption (\ref{final-assumption}) is always satisfied in the framework of Rajput and Rosinski \cite{RajRos}. Therefore, this assumption does not have to be seen as a artificial restrictive assumption, because in the non-signed case (as in \cite{RajRos}) is not a restrictive at all, and in the signed case it is there to satisfy assumption (c) in Theorem \ref{pr-monster}, which comes from the assumption (5) of Theorem 4 in \cite{Horo} and represents a necessary and sufficient condition for the extension of signed bimeasures. The extension of signed bimeasures is essential for the development of the whole theory. This is why we also believe that assumption (\ref{final-assumption}) cannot be weakened. }
\end{rem}
We are now ready to present the results on QID random measures.
\begin{pro}\label{pro-misure-generale}
	Let $\nu_{0}:\mathcal{S}\mapsto\mathbb{R}$ be a signed measure, $\nu_{1}:\mathcal{S}\mapsto\mathbb{R}$ be a measure, $F_{A}$ be a quasi-L\'{e}vy type measure on $\mathbb{R}$ for every $A\in\mathcal{S}$, $\mathcal{S}\ni A\mapsto F_{A}(B)\in(-\infty,\infty)$ be a signed measure for every $B\in\mathcal{B}(\mathbb{R})$ such that $0\notin \overline{B}$ and that $(\nu_{0}(A), \nu_{1}(A),F_{A})$ is the characteristic triplet of a random variable, call it $\Lambda(A)$, $\forall A\in\mathcal{S}$. Assume that $F$ satisfies $(\ref{final-assumption})$ and let $\nu$ be defined as in (\ref{nu}). Moreover, define 	
	\begin{equation}\label{lambda3}
	\lambda(A)=|\nu_{0}|(A)+\nu_{1}(A)+\nu(A).
	\end{equation}
	Then $\lambda:\mathcal{S}\mapsto[0,\infty)$ is a measure s.t.~$\lambda(A_{n})\rightarrow0$ implies $\Lambda(A_{n})\stackrel{p}{\rightarrow}0$ for every $\{A_{n}\}\subset\mathcal{S}$.
\end{pro}
\begin{proof}	
	$\lambda(A)$ is a sum of three measures on $\mathcal{S}$, hence it is a measure on $\mathcal{S}$. Further, let $\lambda(A_{n})\rightarrow0$ for some $\{A_{n}\}\subset\mathcal{S}$, then we have that $|\nu_{0}|$, $\nu_{1}$ and $\nu$ go to zero. Recall $\nu$ satisfies (\ref{xi}). Then, by L\'{e}vy continuity theorem $\Lambda(A_{n})\stackrel{p}{\rightarrow}0$ as $n\rightarrow\infty$.
\end{proof}
\begin{defn}
	Since $\lambda(S_{n})<\infty$, $n=1,2,...$ we extend $\lambda$ to a $\sigma$-finite measure on $(S,\sigma(\mathcal{S}))$; we call $\lambda$ the \textnormal{control measure} of $\Lambda$.
\end{defn}
\begin{lem}\label{l4}
	Let $F_{\cdot}$ be as in Proposition \ref{pro-misure-generale}. There exists a function $\rho:S\times\mathcal{B}_{0}(\mathbb{R})\mapsto \mathbb{R}$ such that
	\\ \textnormal{(i)} $\rho(s,\cdot)$ is a quasi-L\'{e}vy type measure on $\mathcal{B}(\mathbb{R})$, for every $s\in S$, with positive and negative parts denoted by $\rho^{+}(s,\cdot)$ and $\rho^{-}(s,\cdot)$,
	\\ \textnormal{(ii)} $\rho^{+}(\cdot,B)$ and $\rho^{-}(\cdot,B)$, are $\sigma(\mathcal{S})$-measurable functions, for every $B\in\mathcal{B}(\mathbb{R})$,
	\\ Moreover, there exist two unique $\sigma$-finite measures $\tilde{F}^{+}$ and $\tilde{F}^{-}$ on $\sigma(\mathcal{S})\otimes\mathcal{B}(\mathbb{R})$ s.t.
	\begin{equation*}
\int_{S\times\mathbb{R}}h(s,x)\tilde{F}^{+}(ds,dx)=\int_{S}\int_{\mathbb{R}}h(s,x)\rho^{-}(s,dx)\lambda(ds)
	\end{equation*}
	for every $\sigma(\mathcal{S})\otimes\mathcal{B}(\mathbb{R})$-measurable function $h:S\times\mathbb{R}\mapsto[0,\infty]$, and the same holds for $\tilde{F}^{-}$. This equality can be extended to real and complex-valued functions $h$. Finally, for every $A\in\mathcal{S}$ and for every $\mathcal{B}(\mathbb{R})$-measurable real function $g$ s.t.~$\int_{A}\int_{\mathbb{R}}|g(x)||\rho|(s,dx)\lambda(ds)<\infty$, we have that 
	\begin{equation*}
	\int_{\mathbb{R}}g(x)F_{A}(dx)=\int_{A}\int_{\mathbb{R}}g(x)\rho(s,dx)\lambda(ds),
	\end{equation*}
	and for every $B\in\mathcal{B}(\mathbb{R})$ s.t.~$0\notin \overline{B}$,
	\begin{equation*}
	\tilde{F}^{+}(A,B)\geq F_{A}^{+}(B)\quad\text{and}\quad \tilde{F}^{-}(A,B)\geq F_{A}^{-}(B).
	\end{equation*}
\end{lem}
\begin{proof}
	First, notice that $J(A,B)$ satisfies the assumptions of Theorem \ref{pr-monster} with $(T,\mathcal{A})=(S,\mathcal{S})$ and $(X,\mathcal{B})=(\mathbb{R},\mathcal{B}(\mathbb{R}))$.
	Therefore, there exists a finite real valued set function $Q$ on $\bigcup_{D\in\mathcal{S}}(\mathcal{S}\cap D)\otimes\mathcal{B}$ such that
	\begin{equation*}
	Q(A\times B)=J(A,B)=\int_{A}q(s,B)\nu(ds)=\int_{A}q^{+}(s,B)\nu(ds)-\int_{A}q^{-}(s,B)\nu(ds),
	\end{equation*}
	where $q^{+}$ and $q^{-}$ satisfy $(d)'$ and $(e)'$, and $q$ satisfies $(d)$ and $(e)$ of Theorem \ref{pr-monster}. Since $J(A,\{0\})=0$ for every $A\in\mathcal{S}$ and since $q^{+}(s,\cdot)$ and $q^{-}(s,\cdot)$ are mutually singular, we deduce that $q^{+}(s,\{0\})=0$ and $q^{-}(s,\{0\})=0$ $\nu$-a.e.. 
	
	Observe that we can consider $q^{+}(s,\{0\})=0$ and $q^{-}(s,\{0\})=0$ for every $s\in S$. This is because of the following argument. Let $q^{+}(s,\{0\})=0$ $\nu$-a.e.~and let $\tilde{q}^{+}(s,B)=q^{+}(s,B\setminus\{0\})$ for every $B\in\mathcal{B}(\mathbb{R})$. Then $s\mapsto\tilde{q}^{+}(s,B)$ is $\sigma(\mathcal{S})$-measurable since $s\mapsto q^{+}(s,B\setminus\{0\})$ is $\sigma(\mathcal{S})$-measurable, for every $B\in\mathcal{B}(\mathbb{R})$. Moreover, for every sequence of disjoint sets $B_{1},B_{2},...\in\mathcal{B}(\mathbb{R})$ $\tilde{q}^{+}(s,\cup_{i=1}^{\infty} B_{i})=q^{+}(s,\cup_{i=1}^{\infty}B\setminus\{0\})=\sum_{i=1}^{\infty}q^{+}(s,B_{i}\setminus\{0\})=\sum_{i=1}^{\infty}\tilde{q}^{+}(s,B_{i})$. Therefore, $\tilde{q}^{+}$ satisfies the same properties of $q^{+}$, namely  $(d)'$ and $(e)'$ of Theorem \ref{pr-monster}, and $\tilde{q}^{+}(s,\cdot)=q^{+}(s,\cdot)$, off a set of $\nu$-measure zero. The same applies to $q^{-}$ and it is possible to see that $\tilde{q}^{+}(s,\cdot)$ and $\tilde{q}^{-}(s,\cdot)$ are the Jordan decomposition of a signed measure $\tilde{q}(s,\cdot)$, for every $s\in S$, and that $\tilde{q}(s,\cdot)=q(s,\cdot)$, off a set of $\nu$-measure zero. Hence, all the results of Theorem \ref{pr-monster} applied to the present setting remains unchanged (indeed $\tilde{q}$ can be seen as the `$q_{1}$' in the statement of Theorem \ref{pr-monster}). Thus, we consider $q^{+}(s,\{0\})=q^{-}(s,\{0\})=0$ for every $s\in S$.
	
	Since $\lambda\gg\nu$, define
	\begin{equation*}
	\rho^{+}(s,dx):=\frac{d\nu}{d\lambda}(s)(1\wedge x^{2})^{-1}q^{+}(s,dx),\,\,\,\,\textnormal{and}\,\,\,\, \rho^{-}(s,dx):=\frac{d\nu}{d\lambda}(s)(1\wedge x^{2})^{-1}q^{-}(s,dx).
	\end{equation*}
	Thus, $\rho^{+}(\cdot,B)$ and $\rho^{-}(\cdot,B)$ are Borel measurable (precisely $\sigma(\mathcal{S})$-measurable) functions. Further, notice that
	\begin{equation*}
	\int_{\mathbb{R}}(1\wedge x^{2})\rho^{+}(s,dx)=\frac{d\nu}{d\lambda}(s)\int_{\mathbb{R}}q^{+}(s,dx)\leq 1,
	\end{equation*}
	where the last inequality comes from the fact that $\frac{d\nu}{d\lambda}(s)\leq 1$ for all $s\in S$. Hence, $\rho^{+}(s,\cdot)$ is a L\'{e}vy measure on $\mathbb{R}$ for all $s\in S$. The same holds for $\rho^{-}(s,\cdot)$. Further, let
	\begin{equation*}
	\rho(s,B):=\rho^{+}(s,B)-\rho^{-}(s,B)\quad\text{for all $s\in S$, $B\in\mathcal{B}(\mathbb{R})$ s.t.~$0\notin \overline{B}$}.
	\end{equation*}
	Then $\rho(s,\cdot)$ is a quasi-L\'{e}vy type measure by Lemma 2.14 in \cite{Pass}, thus obtaining (i). Using the fact that $\rho^{+}(s,\cdot)$ and $\rho^{-}(s,\cdot)$ are mutually singular for every $B\in\mathcal{B}(\mathbb{R})$ s.t.~$0\notin\overline{B}$, then they are the positive and negative parts of $\rho((s,\cdot))$ for every $s\in S$, and so we obtain (ii). 
	
	Now, let
	\begin{equation}\label{F(C)}
	\tilde{F}^{+}(C)=\int_{S}\int_{\mathbb{R}}\textbf{1}_{C}(s,x)\rho^{+}(s,dx)\lambda(ds),
	\end{equation}
	where $C\in\sigma(\mathcal{S})\otimes\mathcal{B}(\mathbb{R})$, then $\tilde{F}^{+}$ is a well defined measure that satisfies, for every $A\in\mathcal{S}$ and $B\in\mathcal{B}(\mathbb{R})$,
	\begin{equation*}
	\tilde{F}^{+}(A\times B)=\int_{A}\int_{B}\rho^{+}(s,dx)\lambda(ds)	=\int_{A}\int_{B}(1\wedge x^{2})^{-1}q^{+}(s,dx)\xi(ds)
	\end{equation*}
	\begin{equation*}
	=\int_{A\times B}(1\wedge x^{2})^{-1}Q^{+}(ds,dx)	\geq\int_{B}(1\wedge x^{2})^{-1}J^{+}(A,dx)=\int_{B}F^{+}_{A}(dx)=F^{+}_{A}(B),
	\end{equation*}
	where $Q^{+}$ is the positive extension of $Q$ (see Theorem \ref{pr-monster}), thus $Q^{+}$ is a measure on $\sigma(\mathcal{S})\otimes\mathcal{B}(\mathbb{R})$. Concerning $J^{+}(A,dx)$, recall that the notation $M^{+}$ and $M^{-}$ for a bimeasure $M$ stands for the Jordan decomposition of $B\mapsto M(A,B)$ for fixed $A$. The same applies to $\tilde{F}^{-}$. Finally, notice that for any $\mathcal{B}(\mathbb{R})$-measurable real function $g$ s.t.~$\int_{A}\int_{\mathbb{R}}|g(x)||\rho|(s,dx)\lambda(ds)<\infty$ we have
	\begin{equation*}
	\int_{A}\int_{\mathbb{R}}g(x)\rho(s,dx)\lambda(ds)
	=\int_{A}\int_{\mathbb{R}}g(x)\rho^{+}(s,dx)\lambda(ds)-\int_{A}\int_{\mathbb{R}}g(x)\rho^{-}(s,dx)\lambda(ds)
	\end{equation*}
	\begin{equation*}
	=\int_{A\times \mathbb{R}}g(x)(1\wedge x^{2})^{-1}Q(ds,dx)=\int_{\mathbb{R}}g(x)(1\wedge x^{2})^{-1}J(A,dx)=\int_{\mathbb{R}}g(x)F_{A}(dx).
	\end{equation*}
\end{proof}
\begin{rem}
\textnormal{The discussion at the beginning of the proof of Lemma \ref{l4} on the possibility to consider $q^{+}(s,\{0\})=q^{-}(s,\{0\})=0$, for every $s\in S$, is implicit in the proofs of Lemma 2.3 in \cite{RajRos}, and of Lemmas 5.19 and 5.28 in \cite{Pass}. We decided to write it explicitly for the sake of clarity and completeness and because our setting requires more attention to detail.}
\end{rem}
Using the above results, we obtain the following proposition.
\begin{pro}\label{pr-3}
	Under the setting of Proposition \ref{pro-misure-generale}, the characteristic function of $\Lambda(A)$ can be written in the form:
	\begin{equation*}
	\mathbb{E}(e^{i\theta\Lambda(A)})=\exp\left(\int_{A}K(\theta,s)\lambda(ds)\right),\quad \theta\in\mathbb{R},A\in\mathcal{S},
	\end{equation*}
	where
	\begin{equation*}
	K(\theta,s)=i\theta a(s)-\frac{\theta^{2}}{2}\sigma^{2}(s)+\int_{\mathbb{R}}e^{i\theta x}-1-i\theta\tau(x)\rho(s,dx),
	\end{equation*}
	$a(s)=\frac{d\nu_{0}}{d\lambda}(s)$, $\sigma^{2}(s)=\frac{d\nu_{1}}{d\lambda}(s)$ and $\rho$ is given by Lemma \ref{l4}, and $\exp(K(\theta,s))$ is the characteristic function of a QID random variable if it exists. Moreover, we have
	\begin{equation*}
	|a(s)|+ \sigma^{2}(s)+\frac{d\nu}{d\lambda}(s)=1,\quad\quad \text{$\lambda$-a.e.}.
	\end{equation*}
\end{pro}
\begin{proof}
	The first statement follows from Lemma \ref{l4} and the L\'{e}vy-Khintchine formulation of $\Lambda(A)$ (\ref{cf}). The second statement follows from the fact that for every $A\in\mathcal{S}$, we have
	\begin{equation*}
	\int_{A}\left(|a(s)|+\sigma^{2}(s)+\frac{d\nu}{d\lambda}(s)\right)\lambda(ds)=|\nu_{0}|(A)+\nu_{1}(A)+\nu(A)=\lambda(A)=\int_{A}d\lambda(ds).
	\end{equation*}
\end{proof}
Let us recall the definition of $\Lambda$-integrability of a measurable function $f$ (see Definition in \cite{Pass}).
\begin{defn}\label{def-integral}
	Let $f(s)=\sum_{j=1}^{n}x_{j}\mathbf{1}_{A_{j}}(s)$ be a real simple function on $S$, where $A_{j}\in\mathcal{S}$ are disjoint. Then, for every $A\in\sigma(\mathcal{S})$, we define
	\begin{equation*}
	\int_{A}fd\Lambda=\sum_{j=1}^{n}x_{j}\Lambda(A\cap A_{j}).
	\end{equation*}
	Further, a measurable function $f:(S,\sigma(\mathcal{S}))\rightarrow(\mathbb{R},\mathcal{B}(\mathbb{R}))$ is said to be $\Lambda$-integrable if there exists a sequence $\{f_{n}\}$ of simple functions such that 
	\\\textnormal{(i)} $f_{n}\rightarrow f$, $\lambda$-a.e.,
	\\\textnormal{(ii)} for every $A\in\sigma(\mathcal{S})$, the sequence $\{\int_{A}f_{n}d\Lambda\}$ converges in probability as $n\rightarrow\infty$.
	
	If $f$ is $\Lambda$-integrable, then we write
	\begin{equation*}
	\int_{A}fd\Lambda=\mathbb{P}-\lim\limits_{n\rightarrow\infty}\int_{A}f_{n}d\Lambda
	\end{equation*}
	where $\{f_{n}\}$ satisfies $\textnormal{(i)}$ and $\textnormal{(ii)}$.
\end{defn}
As proved in Lemma 5.8 in \cite{Pass} the integral $\int_{A}fd\Lambda$ is well-defined. In the following result we provide a representation for the characteristic function of $\int_{S}fd\Lambda$. The remaining results of this section follow from the same arguments as the ones used in the proof of their respective results in \cite{Pass}. This is because, despite the fact that we are considering a larger class of QID random measures (because our assumptions are weaker), the structure of the control measure and of the representations of $F$, $\tilde{F}^{+}$ and $\tilde{F}^{-}$ are similar to the ones in \cite{Pass}.
\begin{pro}
	Under the setting of Proposition \ref{pro-misure-generale}, if $f$ is $\Lambda$-integrable, then we have that $\int_{S}|K(tf(s),s)|\lambda(ds)<\infty$, where $K$ is given in Proposition \ref{pr-3}, and that
	\begin{equation*}
	\hat{\mathcal{L}}\left(\int_{S}fd\Lambda\right)(\theta)=\exp\left(\int_{S} K(\theta f(s),s)\lambda(ds)\right),\quad \theta\in\mathbb{R}.
	\end{equation*}
\end{pro}
\begin{proof}
	It follows from the same arguments as the ones used in the proof of Proposition 5.9 in \cite{Pass}.
\end{proof}
We state an important result on the integrability conditions of $\int_{S}fd\Lambda$.
\begin{thm}\label{theorem-general}
	Let $f:S\rightarrow\mathbb{R}$ be a $\sigma(\mathcal{S})$-measurable function and consider the setting of Proposition \ref{pro-misure-generale}. Then $f$ is $\Lambda$-integrable if the following three conditions hold:
	\\\textnormal{(i)} $\int_{S}|U(f(s),s)|\lambda(ds)<\infty$,
	\\\textnormal{(ii)} $\int_{S}|f(s)|^{2}\sigma^{2}(s)\lambda(ds)<\infty$,
	\\\textnormal{(iii)} $\int_{S}V_{0}(f(s),s)\lambda(ds)<\infty$,\\
	where
	\begin{equation*}
U(u,s)=ua(s)+\int_{\mathbb{R}}\tau(xu)-u\tau(x)\rho(s,dx),\,\,\,\,\text{and}\,\,\,\, V_{0}(u,s)=\int_{\mathbb{R}}(1\wedge |xu|^{2})|\rho|(s,dx).
	\end{equation*}
	Further, the characteristic function of $\int_{S}fd\Lambda$ can be written as
	\\\textnormal{(iv)} $\hat{\mathcal{L}}\left(\int_{S}fd\Lambda\right)(\theta)=\exp\left(i\theta a_{f}-\frac{1}{2}\theta^{2}\sigma_{f}^{2}+\int_{\mathbb{R}}e^{i\theta x}-1-i\theta\tau(x)F_{f}(dx) \right)$,
	\begin{equation*}
	\textnormal{where}\quad a_{f}=\int_{S}U(f(s),s)\lambda(ds),\quad \sigma_{f}^{2}=\int_{S}|f(s)|^{2}\sigma^{2}(s)\lambda(ds),\quad\textnormal{and}
	\end{equation*}
	$F_{f}(B)$ is the unique quasi-L\'{e}vy measure determined by the difference of the L\'{e}vy measures $\tilde{F}^{+}_{f}$ and $\tilde{F}^{-}_{f}$, which are defined as: for every $B\in\mathcal{B}(\mathbb{R})$
	\begin{equation*}
	\tilde{F}^{+}_{f}(B)=\tilde{F}^{+}(\{(s,x)\in S\times\mathbb{R}:\, f(s)x\in B\setminus\{0 \} \})\quad\text{and}
		\end{equation*}
	\begin{equation*}
	\tilde{F}^{-}_{f}(B)=\tilde{F}^{-}(\{(s,x)\in S\times\mathbb{R}:\, f(s)x\in B\setminus\{0 \} \}).
	\end{equation*}
\end{thm}
\begin{proof}
	It follows from the same arguments as the ones used in the proof of Theorem 5.10 in \cite{Pass}. 
\end{proof}
We conclude with a result on the continuity of the stochastic integral mapping. Before presenting the result, we need some preliminaries. Define the Musielak-Orlicz space as in \cite{RajRos}:
\begin{equation*}
L_{\Phi_{p}}(S;\lambda)=\left\{f\in L_{0}(S;\lambda):\,\,\,\int_{S}\Phi_{p}(|f(s)|,s)\lambda(ds)<\infty \right\}.
\end{equation*}
By an equivalent result of Lemma 6.1 in \cite{Pass} applied here it is possible to see that the space $L_{\Phi_{p}}(S;\lambda)$ is a complete linear metric space with the $F$-norm defined by 
\begin{equation*}
\| f\|_{\Phi_{p}}=\inf\limits_{c>0}\left\{\int_{S}\Phi_{p}(c^{-1}|f(s)|,s)\lambda(ds)\leq c \right\}.
\end{equation*}
Simple functions are dense in $L_{\Phi_{p}}(S;\lambda)$ and $L_{\Phi_{p}}(S;\lambda)\hookrightarrow L_{0}(S;\lambda)$ is continuous, where in the present case $L_{0}(S;\lambda)$ is equipped with the topology of convergence in $\lambda$ measure on every set of finite $\lambda$-measure. Moreover, we have $\| f_{n}\|_{\Phi_{p}}\rightarrow0\Leftrightarrow \int_{S}\Phi_{p}(|f(s)|,s)\lambda(ds)\rightarrow0$. Now, define, for $0\leq p\leq q$, $u\in\mathbb{R}$ and $s\in S$,
\begin{equation}\label{Phi}
\Phi_{p}(u,s)=U^{*}(u,s)+u^{2}\sigma^{2}(s)+V_{p}(u,s),
\end{equation}
where
\begin{equation*}
U^{*}(u,s)=\sup\limits_{|c|\leq 1}|U(cu,s)|\,\,\,\,\textnormal{and}\,\,\,\, V_{p}(u,s)=\int_{\mathbb{R}}|xu|^{p}\textbf{1}_{|xu|>1}(x) +|xu|^{2}\textbf{1}_{|xu|\leq 1}(x)|\rho|(s,dx).
\end{equation*}
\begin{thm}\label{Orlicz-theorem}
	Let $0\leq p\leq q$ and $\Phi_{p}$ defined as in $(\ref{Phi})$. Then
	\begin{equation*}
	\left\{\textnormal{$f:$ $f$ is $\Lambda$-integrable and }\mathbb{E}\left[\left|\int_{S}fd\Lambda\right|^{p}\right]<\infty\right\}\supset L_{\Phi_{p}}(S;\lambda)
	\end{equation*}
	and the linear mapping
	\begin{equation*}
	L_{\Phi_{p}}(S;\lambda)\ni f\mapsto\int_{S}fd\Lambda\in L_{p}(\Omega;\mathbb{P})
	\end{equation*}
	is continuous.
\end{thm}
\begin{proof}
	It follows from the same arguments as the ones used in the proof of Theorem 6.3 in \cite{Pass}. 
\end{proof}


\begin{thebibliography}{0}
	\bibitem{Khartov} 	I. Alexeev, A. Khartov A criterion and a Cram\'{e}r–Wold device for quasi-infinite divisibility for discrete multivariate probability laws.
	\textit{Electron. J. Probab.} 28: 1-17 (2023).
	\bibitem{Lindner} Berger, D. and Lindner, A. 
	A Cram\'{e}r-Wold device for infinite divisibility of $Z^d$-valued distributions.
	\textit{Bernoulli }28, 2, 1276-1283 (2022).
	\bibitem{Blei} Blei, R.C. Multilinear measure theory and the Grothendieck factorization theorem.	\textit{Proc. Lond. Math. Soc.} (3), 56 (1988), 529-546.
	\bibitem{Blei2}Blei, R.C. Multi-linear measure theory and multiple stochastic integration,\textit{ Probability Theory and Related Fields}, 81 (1989), 569-584.
	\bibitem{Bowers} Bowers, A. A Radon--Nikodym theorem for Fr\'{e}chet measures. \textit{Journal of Mathematical Analysis and Applications}, 411, (2), 592-606 (2014).	
	\bibitem{Physics2} Chaiba H., Demni, N., Mouayn, Z. Analysis of generalized negative binomial distributions attached to hyperbolic Landau levels. \textit{Journal of Mathematical Physics} 57, 072-103 (2016).
	\bibitem{Daley} Daley D.J. and Vere-Jones D. \textit{An Introduction to the	Theory of Point Processes Volume II: General Theory and Structure},	Second Edition, Springer (2008)
	\bibitem{Physics1} Demni, N., Mouayn, Z. Analysis of generalized  Poisson  distributions  associated  with  higher  Landau  levels.  \textit{Infin.~Dimens.~Anal.~Quantum Probab.~Relat.~Top.}, 18(04), (2015).
	\bibitem{Dobrakov} Dobrakov, I.	Multilinear integration of bounded scalar valued functions. \textit{Math. Slovaca}, 49 (1999), pp. 295-304
	\bibitem{Frechet} Fr\'{e}chet, M. Sur les fonctionnelles bilin\'{e}aires. \textit{Trans. Amer. Math. Soc.}, 16 (1915), 215-234.
	\bibitem{Grothe} Grothendieck, A.: Produits tensorieis topologiques et espaces nucl\'{e}aires. \textit{Memoirs Amer. Math. Soc.} 16 (1955)
	\bibitem{Horo}Horowitz J. Une remarque sur les bimesures. \textit{S\'{e}minaire de probabilit\'{e}s}, 11 (1977), 59-64.
	\bibitem{Kallenberg2} Kallenberg O. \textit{Random Measures, Theory and Applications}. Springer, (2017).
	\bibitem{Karni} Karni, S., Merzbach, E. On the extension of bimeasures. \textit{J. Anal. Math.} 55, 1-16 (1990).
	\bibitem{Klenke} Klenke A. \textit{Probability Theory: A Comprehensive Course.} Second Edition, Springer, 2014.
	\bibitem{LPS} Lindner A., Pan L., Sato K. On quasi-infinitely divisible distribution, \textit{Trans. Amer. Math. Soc.}, 370, 8483-8520 (2018).
	\bibitem{Naka} Nakamura, T. A complete Riemann zeta distribution and the Riemann hypothesis. \textit{Bernoulli}, 21(1), 604-617, (2015).
	\bibitem{Naka2} Nakamura, T. Zeta distributions generated by Dirichlet series and their (quasi) infinitely divisibility, \textit{private communication}.
	\bibitem{Nau} Naulet, Z. and Barat, E. Some Aspects of Symmetric Gamma Process Mixtures. \textit{Bayesian Analysis}, 13, 3, 703-720, (2018)
	\bibitem{Pass} Passeggeri, R. Spectral representations of quasi-infinitely divisible processes. \textit{Stochastic processes and their applications}, 130, (3), (2020), 1735-1791.
	\bibitem{PassRM} Passeggeri, R. Quasi-infinitely random measures. \textit{Bayesian Analysis}, 18(1): 253-286 (2023).
	\bibitem{PassRSM} Passeggeri, R. Random signed measures. \textit{ArXiv:2411.12623}, (2024).
	\bibitem{PrekopaI}Pr\'{e}kopa, A.: On stochastic set functions I. \textit{Acta Math. Acad. Sci. Hung.} 7, 215-263 (1956).
	\bibitem{RajRos}Rajput, B.S. and Rosinski, J., Spectral representations of infinitely divisible processes. \textit{Probability Theory and Related Fields}, 82, 451-487 (1989).
	\bibitem{SamTaq}Samorodnitsky G., Taqqu M.S. \textit{Stable Non-Gaussian Random Processes: Stochastic Models with Infinite Variance}, Chapman and Hall/CRC, (1994).
	\bibitem{Ylinen} Ylinen, K. On vector bimeasures. \textit{Ann. Mat. Pura Appl.} (4), 117 (1978), pp. 115-138.
	\bibitem{Ylinen2} Ylinen, K. Representations of bimeasures
	\textit{Studia Math.}, 104 (1993), pp. 269-278.
	\bibitem{Zhang} Zhang, H., Liu, Y., Li, B.. Notes on discrete compound Poisson model with applications to risk theory. \textit{Insurance Math. Econom.} 59, 325-336, (2014).

\end{thebibliography}
\end{document}